\documentclass[12pt, reqno]{amsart}
\usepackage{amsmath, amsthm, amscd, amsfonts, amssymb, graphicx, color}
\usepackage[bookmarksnumbered, colorlinks, plainpages]{hyperref}
\hypersetup{colorlinks=true,linkcolor=red, anchorcolor=green, citecolor=cyan, urlcolor=red, filecolor=magenta, pdftoolbar=true}
\input{mathrsfs.sty}

\newtheorem{theorem}{Theorem}[section]

\newtheorem{proposition}[theorem]{Proposition}
\newtheorem{corollary}[theorem]{Corollary}
\theoremstyle{definition}
\newtheorem{definition}[theorem]{Definition}
\newtheorem{example}[theorem]{Example}

\theoremstyle{remark}
\newtheorem{remark}[theorem]{Remark}
\numberwithin{equation}{section}

\begin{document}

\title[Hilbert $C^*$-module independence]{Hilbert $C^*$-module independence}

\author[R. Eskandari, J. Hamhalter, M. S. Moslehian, V. M. Manuilov]{R. Eskandari$^1$, J. Hamhalter$^2$ M. S. Moslehian$^3$, \MakeLowercase{and} V. M. Manuilov$^4$}

\address{$^1$Department of Mathematics, Faculty of Science, Farhangian University, Tehran, Iran.}
\email{eskandarirasoul@yahoo.com}

\address{$^2$Department of Mathematics, Faculty if Electrical Engineering, Czech Technical University in Prague, Czech Republic}
\email{hamhalte@fel.cvut.cz}

\address{$^{3}$Department of Pure Mathematics, Center of Excellence in Analysis on Algebraic Structures (CEAAS), Ferdowsi University of Mashhad, P. O. Box 1159, Mashhad 91775, Iran.}
\email{moslehian@um.ac.ir; moslehian@yahoo.com}

\address{$^4$ Moscow Center for Fundamental and Applied Mathematics, and Department of Mechanics and Mathematics, Moscow State University, Moscow, 119991, Russia.} 
\email{manuilov@mech.math.msu.su}

\renewcommand{\subjclassname}{%
\textup{2020} Mathematics Subject Classification}
\subjclass[]{46L08, 46L05, 47A62.}

\keywords{Hilbert $C^*$-module; $C^*$-independence; state; determining element; module independence.}

\begin{abstract}
We introduce the notion of Hilbert $C^*$-module independence: Let $\mathscr{A}$ be a unital $C^*$-algebra and let $\mathscr{E}_i\subseteq \mathscr{E},\,\,i=1, 2$, be ternary subspaces of a Hilbert $\mathscr{A}$-module $\mathscr{E}$. Then $\mathscr{E}_1$ and $\mathscr{E}_2$ are said to be Hilbert $C^*$-module independent if there are positive constants $m$ and $M$ such that for every state $\varphi_i$ on $\langle \mathscr{E}_i,\mathscr{E}_i\rangle,\,\,i=1, 2$, there exists a state $\varphi$ on $\mathscr{A}$ such that
\begin{align*}
m\varphi_i(|x|)\leq \varphi(|x|) \leq M\varphi_i(|x|^2)^{\frac{1}{2}},\qquad \mbox{for all~}x\in \mathscr{E}_i, i=1, 2.
\end{align*}
We show that it is a natural generalization of the notion of $C^*$-independence of $C^*$-algebras.   Moreover, we demonstrate that even in case of $C^*$-algebras this concept of independence is new and  has a nice characterization in terms of extensions. This enriches the theory of independence of $C^*$-algebras. We show that if  $\langle \mathscr{E}_1,\mathscr{E}_1\rangle $ has the quasi extension property and $z\in \mathscr{E}_1\cap \mathscr{E}_2$ with $\|z\|=1$, then $|z|=1$. Several characterizations of Hilbert $C^*$-module independence and a new characterization of $C^*$-independence are given. One of characterizations states that if $z_0\in \mathscr{E}_1\cap \mathscr{E}_2$ is such that $\langle z_0,z_0\rangle=1$, then $\mathscr{E}_1$ and $\mathscr{E}_2$ are Hilbert $C^*$-module independent if and only if  $\|\langle x,z_0\rangle\langle y,z_0\rangle\|=\|\langle x,z_0\rangle\|\,\|\langle y,z_0\rangle\|$ for all $x\in \mathscr{E}_1$ and $y\in \mathscr{E}_2$. We also provide some technical examples and counterexamples to illustrate our results.
\end{abstract}

\maketitle
%=============================================%
%=============================================%
%=============================================%

\section{Introduction}

Throughout the paper, we assume that $\mathscr{A}$ is a unital $C^*$-algebra with unit $1$. The set $\{a^*a: a\in\mathscr{A}\}$ consisting of all positive elements of $\mathscr{A}$ is denoted by $\mathscr{A}^+$. By a state, we mean a bounded linear functional $\varphi$ on $\mathscr{A}$ of norm-one taking positive elements to positive numbers. The convex set of all states of $\mathscr{A}$ is denoted by $\mathcal{S}(A)$, which is a compact Hausdorff space in its relative weak$^*$-topology inherited from the dual of $\mathscr{A}$. The restriction of a state $\varphi$ to a subalgebra $\mathscr{B}$ is denoted by $\varphi|_\mathscr{B}$.

We employ without further comment the following fundamental facts about elements and states of a $C^*$-algebra $\mathscr{A}$:

\begin{itemize}
\item[(i)]\label{(i)} Given a state $\varphi$, the Cauchy--Schwarz inequality 
\begin{align}\label{cs}
|\varphi(b^*a)|^2\leq \varphi(a^*a)\varphi(b^*b),\qquad a,b\in\mathscr{A},
\end{align}
holds. In particular, if $a$ is self-adjoint, then we get the Kadison inequality:
\begin{align}\label{kad}\varphi(a)^2\leq \varphi(a^2).\end{align}

\item[(ii)] For each  element $a\in \mathscr{A}$ and each state $\varphi$ on $\mathscr{A}$, it follows from \eqref{cs} with $b=1$ that the Choi inequality holds:
\begin{align}\label{choi}
|\varphi(a)|^2\leq \varphi(a^*a).
\end{align}
If equality holds in the later inequality, then we say that $\varphi$ is \emph{definite} at $a$. In this case, for all $b\in\mathscr{A}$, it follows from \eqref{cs} that
\[\left|\varphi\big((a-\varphi(a)1)b\big)\right|^2\leq \varphi\big((a^*-\overline{\varphi(a)}1)(a-\varphi(a)1)\big)\,\varphi(b^*b)=0,\]
whence
$\varphi(ab)=\varphi(a)\varphi(b)$ and similarly $\varphi(ba)=\varphi(a)\varphi(b)$. Thus $\varphi$ has a kind of multiplicative property; see \cite[p. 22]{HAM3}.

\item[(iii)] If $a\in \mathscr{A}$ is self-adjoint and $\varphi$ is a state on $\mathscr{A}$, then it follows from $|a|\leq \|a||$ that $|a|^2\leq \|a\|\,|a|$, whence 
\begin{align}\label{a2}\varphi(a^2) \leq \|a\| \varphi(|a|).\end{align}

\item[(iv)] If $a\in \mathscr{A}$ is self-adjoint, then there is a state $\varphi$ on $\mathscr{A}$ such that $\varphi(a)=\|a\|$. If this occurs, it follows from the Kadison inequality that $\varphi(a^2)\leq \|a^2\|=\|a\|^2=\varphi(a)^2 \leq \varphi(a^2)$. Therefore, $\varphi$ is definite at $a$. It implies that $\varphi$ is multiplicative on the  $C^*$-algebra  generated by $a$.  Especially, $\varphi(a^n)=\varphi(a)^n$.
%      This argument shows that $\varphi(a^{2^n})=\|a^{2^n}\|$, hence $\varphi$ is definite at $a^{2^n}$ for all $n\in\mathbb{N}$ as well. 
In particular, if $a\geq 0$ and $\varphi(a)=\|a\|=1$, then $\varphi(a^2)=\varphi(a^{\frac{1}{2}})=1$.
\end{itemize}
There are several noncommutative versions of the independence in the classical probability spaces. One of them is the notion of $C^*$-independence, which was first introduced by Haag and Kastler \cite{HK} to formulate the degrees of independence between local
physical systems in the framework of algebraic quantum field theory; see also \cite{RED}. 

Let $\mathscr{A}_1$ and $\mathscr{A}_2$ be $C^*$-subalgebras of a $C^*$-algebra $\mathscr{A}$ and let all algebras have the same identity. These subalgebras are called \emph{$C^*$-independent (or statistical independent)} if for every state $\varphi_1$ of $\mathscr{A}_1$ and every state $\varphi_2$ of $\mathscr{A}_2$, there exists a state $\varphi$ of $\mathscr{A}$  extending both $\varphi_1$ and $\varphi_2$. Florig and Summers \cite{FLS} proved that the $C^*$-independence is equivalent to the statement that for every positive norm-one elements $a\in\mathscr{A}_1$ and $b\in \mathscr{A}_2$, there is a state $\varphi$ of $\mathscr{A}$ such that $\varphi(a)=\varphi(b)=1$. Furthermore, they showed that $\mathscr{A}_1$ and $\mathscr{A}_2$ are $C^*$-independent if and only if $\|ab\|=\|a\|\,\|b\|$ for all $a\in\mathscr{A}_1$ and $b\in \mathscr{A}_2$.\\
There are a vast range of examples of nonabelian and abelian subalgebras that are (are not) $C^*$-independent; see \cite[Chapter 11]{HAM3}.

There is another important type of independence. Two $C^*$-subalgebras $\mathscr{A}_1$ and $\mathscr{A}_2$ are called \emph{Schlieder independent}, briefly \emph{$S$-independent}, if $ab\neq 0$ whenever $a\in\mathscr{A}_1$ and $b\in \mathscr{A}_2$ are nonzero; see \cite{SCH}. It is known that
\[C^*\mbox{-independence} \Longrightarrow S\mbox{-independence}\]
and the reverse implications are valid if the $C^*$-subalgebras $\mathscr{A}_1$ and $\mathscr{A}_2$ elementwisely commute; see \cite{ROO, HAM1}.

Another significant noncommutative notion of independence is that of operationally $C^*$-independence. Two $C^*$-subalgebras $\mathscr{A}_1$ and $\mathscr{A}_2$ of a $C^*$-algebra $\mathscr{A}$ are said to be \emph{operationally $C^*$-independent} if any two completely positive unit preserving maps $T_1: \mathscr{A}_1\to \mathscr{A}_1$ and $T_2: \mathscr{A}_2\to \mathscr{A}_2$ have a joint extension to a completely positive unit preserving map $T:\mathscr{A} \to \mathscr{A}$. Evidently, the operationally $C^*$-independency implies the $C^*$-independency of $C^*$-subalgebras; see \cite{RS, HAM3}.

The concept of a Hilbert $C^*$-module is a generalization of that of a Hilbert space in which the inner product takes its values in a $C^*$-algebra instead of the field of complex numbers. However, some basic properties of Hilbert spaces are no longer valid in the framework of Hilbert $C^*$-modules, in general. For instance, not any closed submodule of a Hilbert $C^*$-module is orthogonally complemented, and not every bounded $C^*$-linear operator on a Hilbert $C^*$-module is adjointable. Furthermore, a $C^*$-algebra $\mathscr{A}$ can be regarded as a Hilbert module over itself via $\langle a,b\rangle=a^*b\,\,(a, b \in \mathscr{A})$. The norm on a (right) Hilbert $C^*$-module $(\mathscr{E}, \langle\cdot,\cdot\rangle)$ over a $C^*$-algebra $\mathscr{A}$ is defined by $\|x\|:=\|\langle x, x\rangle\|^{\frac{1}{2}}$. The positive square root of $\langle x, x\rangle$ in $\mathscr{A}$ is denoted by $|x|$ for $x\in \mathscr{E}$. For a closed subspace $\mathscr{F}$ of $\mathscr{E}$, we denote by $\langle\mathscr{F},\mathscr{F}\rangle$ the closed linear span of $\{\langle x,y\rangle: x, y\in \mathscr{F}\}$. 

The reader is referred to \cite{R} for more information on the theory of $C^*$-algebras and to \cite{LAN, MT, FRA1} for terminology and notation on Hilbert $C^*$-modules.

The papers is organized as follows. In the next section, we introduce the notion of Hilbert $C^*$-module independence of ternary spaces by employing states on considered $C^*$-algebras. We prove that the states can be replaced by pure states. In addition, we show that it is a natural generalization of the notion of $C^*$-independence in the category of unital $C^*$-algebras. Furthermore, we establish a module counterpart of the fact that if $\mathscr{A}_1$ and $\mathscr{A}_2$ are $C^*$-independent, then $\mathscr{A}_1\cap \mathscr{A}_2=\mathbb{C}$, by showing that if $\langle \mathscr{E}_1,\mathscr{E}_1\rangle $ has the quasi extension property and $z\in \mathscr{E}_1\cap \mathscr{E}_2$ with $\|z\|=1$, then $|z|=1$. In Section 3, we provide several characterizations of Hilbert $C^*$-module independence and give a new characterization of $C^*$-independence. One of characterizations states that if $z_0\in \mathscr{E}_1\cap \mathscr{E}_2$ is such that $\langle z_0,z_0\rangle=1$, then $\mathscr{E}_1$ and $\mathscr{E}_2$ are Hilbert $C^*$-module independent if and only if  $\|\langle x,z_0\rangle\langle y,z_0\rangle\|=\|\langle x,z_0\rangle\|\,\|\langle y,z_0\rangle\|$ for all $x\in \mathscr{E}_1$ and $y\in \mathscr{E}_2$. In Section 4, we put a $C^*$-structure on a certain quotient module $[E]$ and show that a natural assumption, the Hilbert $C^*$-module independence is equivalent to $C^*$-independence of particular $C^*$-subalgebras of $[E]$. In the last section, we provide some technical examples to illustrate our results.

%=============================================%
%=============================================%
%=============================================%

\section{Module independence and $C^*$-independence}

 We start our work with the case of  $C^*$-algebras.  We propose a new concept of independence  of $C^*$-algebras, called module independence, that is more general than the classical $C^*$-independence. This concept  captures the independence of subalgebras not in terms of extensions but rather  in terms of  inequalities and  kernel spaces.
 
%=============================================%    
\begin{definition}
Let $\mathscr{A}_1$ and  $\mathscr{A}_2$ be unital $C^*$-subalgebras of a unital $C^*$-algebra $\mathscr{A}$. We say that  $\mathscr{A}_1$ and  $\mathscr{A}_2$ are module independent if the following condition holds. There are positive constants $m$ and $M$ such that for every pair of states $\varphi_1$ and $\varphi_2$  on
$\mathscr{A}_1$ and $\mathscr{A}_2$, respectively, there is a state $\varphi$ on $\mathscr{A}$ such that the following inequalities are true:
      
\begin{align}\label{IN1}
m\varphi_i(|x|)\leq \varphi(|x|) \leq M\varphi_i(|x|^2)^{\frac{1}{2}},\qquad \mbox{for all~}x\in \mathscr{A}_i, i=1, 2.
\end{align} 
\end{definition}
%=============================================%

We now give some characterization of this concept in terms of kernel spaces and conditional extensions.  
Given a state $\varphi$ on a $C^*$-algebra $\mathscr{A}$, we  denote its left kernel by
  \[  \mathscr{N}_\varphi =\{x\in \mathscr{A} \, |\,    \varphi(x^*x)=0\}  . \]
 It is well known that $\mathscr{N}_\varphi\subseteq \ker \varphi$ and that the positive parts of the kernel and left kernel coincide. 
%=============================================%
\begin{theorem}\label{I1}
Let $\mathscr{A}_1$ and  $\mathscr{A}_2$ be unital
$C^*$-subalgebras of a unital $C^*$-algebra  
$\mathscr{A}$.	(We do not assume that algebras in question have common unit.) Then the following statements are equivalent: 
\begin{enumerate}
	\item  $\mathscr{A}_1$ and  $\mathscr{A}_2$ are module independent in 
	$\mathscr{A}$. 
	\item There is a positive constant  $m$ such that  for each couple of states $\varphi_1$ in $ S(\mathscr A_1)$ and  
	$\varphi_2$  in $S(\mathscr{A}_2)$, there is a state $\varphi\in S(\mathscr{A})$  with 
	\[  \varphi(1_{\mathscr{A}_1}),    \varphi(1_{\mathscr{A}_2})\ge m     \]
such that 
	\[ \mathscr{N}_{\varphi|\mathscr{A}_1} =        \mathscr{N}_{\varphi_1}, \qquad 
	 \mathscr{N}_{\varphi|\mathscr{A}_2} =        \mathscr{N}_{\varphi_2}   . \] 
	
	\item  There is a constant $m>0$ such that  for each couple of pure states $\varphi_1$ and $\varphi_2$ on $\mathscr{A}_1$ and $\mathscr{A}_2$, respectively,   there is a state $\varphi$ on $\mathscr{A}$  with
		\[  \varphi(1_{\mathscr{A}_1}),    \varphi(1_{\mathscr{A}_2})\ge m     \]
 	such that the restriction of $\varphi$ to $\mathscr{A}_1$ and $\mathscr{A}_2$ is a scalar multiple  of $\varphi_1$ and $\varphi_2$, respectively. 
	\item There is a constant $m>0$ such that  for each couple of  states $\varphi_1$ and $\varphi_2$ on $\mathscr{A}_1$ and $\mathscr{A}_2$, respectively,   there is a state $\varphi$ on $\mathscr{A}$  with
	\[  \varphi(1_{\mathscr{A}_1}),    \varphi(1_{\mathscr{A}_2})\ge m     \]
 such that the restriction of $\varphi$ to $\mathscr{A}_1$ and $\mathscr{A}_2$ is a scalar multiple  of $\varphi_1$ and $\varphi_2$, respectively. 
\end{enumerate} 
\end{theorem}
\begin{proof}
	(1) implies (2): Suppose that $m$ and $M$ are constants from equation (\ref{IN1}). Suppose that there are given states $\varphi_1$ and $\varphi_2$ on $\mathscr{A}_1$ and $\mathscr{A}_2$, respectively. Let $\varphi$  be a state on $\mathscr{A}$ fulfilling (\ref{IN1}). Take now 
	$x\in\mathscr{N}_{\varphi_1}$.  It means that $\varphi_1(x^*x)=\varphi_1(|x|^2)=0$. By (\ref{IN1}), we obtain $\varphi(|x|)=0$.  Therefore $\varphi(|x|^2)=0$. In other words,  we have 
	$\mathscr{N}_{\varphi_1}\subseteq \mathscr{N}_{\varphi}$.
	
	On the other hand, let us pick $x\in \mathscr{A}_1$ with $\varphi(|x|^2)=0$. Again using (\ref{IN1}), we obtain $\varphi(|x|^2)=0\ge m \varphi_1(|x|^2)$ and so   $\varphi_1(|x|^2)=0$.      It says that $\mathscr{N}_{\varphi}\cap \mathscr{A}_1\subseteq \mathscr{N}_{\varphi_1}$. 
Employing the same reasoning for $\varphi_2$, we can see that 
		\[ \mathscr{N}_{\varphi|\mathscr{A}_1} =        \mathscr{N}_{\varphi_1}            \] 
	and 
	\[ \mathscr{N}_{\varphi|\mathscr{A}_2} =        \mathscr{N}_{\varphi_2}   . \] 
	
	Now we prove that (2) implies (3). 
	Let us fix pure states $\mathscr{\varphi}_1$ and 	$\mathscr{\varphi}_2$ on $\mathscr{A}_1$ and 	$\mathscr{A}_2$, respectively. Let $\varphi$ be a state on $\mathscr{A}$ satisfying condition (2). Then also $\mathscr{N}^*_{\varphi_1}= \mathscr{N}^*_{\varphi}\cap \mathscr A_1$. By the property of pure states \cite[Theorem 5.3.4]{Pedersen}, we know that
	\[ \ker\,  \varphi_1=\mathscr{N}_{\varphi_1}+\mathscr{N}^*_{\varphi_1} \subseteq \mathscr{N}_\varphi+\mathscr{N}^*_\varphi\subseteq \ker\,  \varphi.      \] 
Simple linear algebra says that $\varphi$ on $\mathscr{A}_1$ is   a scalar multiple of $\varphi_1$. The same conclusion holds for $\varphi_2$. 	\\

Let us now prove (3) implies (4). Let $m$ be a constant from condition (3). It will be convenient to introduce the set 
  \[ F_m = \{\varphi\in S(\mathscr{A})\,|\, \varphi(1_{\mathscr A_1}), \varphi(1_{\mathscr A_2})\ge m       \} .  \] 
It is a weak$^*$-closed face of the state space of $\mathscr{A}$. 
We shall first prove (4) in a special case when $\varphi_1$ is a convex combination of pure states and $\varphi_2$ is a pure state. More specifically, let 
\[  \varphi_1=\sum_{i=1}^n \alpha_i\varphi_{1,i},   \]
where all $\varphi_{1,i}$ are pure and $\alpha_i\ge 0$ with $\sum_{i=1}^n\alpha_i=1$. Let $\varphi_2$ be pure. From (2), for each $i$,   we can find a   state $\tilde{\varphi}_{1,i}$ in $F_m$  such that each state  $\tilde{\varphi}_{1,i}$ is a multiple of state $\varphi_{1,i}$ on $\mathscr{A}_1$ and multiple of state $\varphi_2$ on $\mathscr{A}_2$. Taking now the convex combination 
\[  \varphi= \sum_{i=1}^n \alpha_i   \tilde{\varphi}_{1,i},    \]   
we obtain a  state in $F_m$ that coincides with multiples of $\varphi_1$ and  $\varphi_2$ on
 $\mathscr{A}_1$ and $\mathscr{A}_2$, respectively. Suppose now that the state $\varphi_1$ is of the same form as before and we have more general state 
 \[   \varphi_2=\sum_{j=1}^m \beta_j\varphi_{2,j},   \]            
where all $\varphi_{2,j}$ are pure and $\beta_j\ge 0$ with $\sum_{j=1}^m\beta_j=1$.
By the same reasoning as above, for each $j$, we can find a  state 
$\tilde{\varphi}_{2,j}$ in $F_m$ that coincides with multiple of  $\phi_1$ on $\mathscr{A}_1$ and 
multiple of $\varphi_{2,j}$ on $\mathscr{A}_2$. Again taking 
\[  \varphi= \sum_{j=1}^m \beta_j \tilde{\varphi}_{2,j},             \]
we obtain a state in $F_m$ that restricts to  multiples of $\varphi_1$ and $\varphi_2$ on $\mathscr{A}_1$ and  $\mathscr{A}_2$, respectively.\\

Let us now consider general states $\varphi_1$ and $\varphi_2$ on $\mathscr{A}_1$ and $\mathscr{A}_2$, respectively. By the Krein--Milman theorem, there are nets of states 
$(\varphi_{1,\alpha})_\alpha$ and $(\varphi_{2,\alpha})_\alpha$ on $\mathscr{A}_1$ and  $\mathscr{A}_2$, respectively,  with the following properties:
  \begin{enumerate}
  	\item Each $\varphi_{1,\alpha}$ and $\varphi_{2,\alpha}$ is a convex combination of finitely many pure states. 
  	\item The nets $(\varphi_{1,\alpha})\alpha$ and $(\varphi_{2,\alpha})\alpha$ converge in the weak$^*$-topology to $\varphi_1$ and $\varphi_2$ on the corresponding subalgebras, respectively. 
  	\end{enumerate}
By the results in the previous part,  we can find a net of states $(\varphi_\alpha)_\alpha$ in $F_m$ such that, for each $\alpha$, we have 
\[ \varphi_\alpha(a)= \varphi_\alpha(1_{\mathscr{A}_1}) \varphi_{1,\alpha}(a) \qquad a\in \mathscr{A}_1,                                       \] 

\[ \varphi_\alpha(a)= \varphi_\alpha(1_{\mathscr{A}_2}) \varphi_{2,\alpha}(a) \qquad a\in \mathscr{A}_2,                                        \] 

and 

\[  m\le \varphi_\alpha(1_{\mathscr{A}_1}), \varphi_\alpha(1_{\mathscr{A}_2})\le 1.    \]

Now using the compactness of  $F_m$ and $[m,1]$ and  taking subnets, we can find  a state $\varphi$ on $\mathscr A$  that is in $F_m$ that restricts to  multiples of $\varphi_1$ and $\varphi_2$ on $\mathscr{A}_1$ and  $\mathscr{A}_2$, respectively.\\

It remains to show that (4) implies (1).  Suppose (4) holds, and take now  states $\varphi_1$ and $\varphi_2$ on $\mathscr{A}_1$ and  $\mathscr{A}_2$, respectively. From (4), there is a state $\varphi$ on $\mathscr{A}$ such that 
\[ \varphi(a)= \varphi(1_{\mathscr{A}_1}) \varphi_{1}(a)  \qquad a\in \mathscr{A}_1,                                        \] 

\[ \varphi(a)= \varphi(1_{\mathscr{A}_2}) \varphi_{2}(a)  \qquad a\in \mathscr{A}_2,                                         \] 
where $1\ge \varphi(1_{\mathscr{A}_1}), \varphi(1_{\mathscr{A}_2})\ge m>0.$ This gives 
   \[ m\varphi_1(|x|)\le \varphi(1_{\mathscr{A}_1}) \varphi_1(|x|)=\varphi(|x|)\le \varphi_1(|x|)\le  \varphi_1(|x|^2)^{\frac{1}{2}}\,,   \qquad x\in \mathscr{A}_1.\]

The same inequality holds for state $\varphi_2$.
\end{proof}
%=============================================%
Even if the module independence  is close to $C^*$-independence, as the following simple  example shows, it is strictly weaker.\\
%=============================================%
\begin{example}\label{EX10}
	Let $\mathscr{A}$ be the algebra of two-by-two complex matrices. Take two orthogonal one-dimensional projections $p_1$ and $p_2$ in $\mathscr{A}$. Consider subalgebras $\mathscr{A}_1=p_1\mathscr{A}p_1$ 
and 	$\mathscr{A}_2=p_2\mathscr{A}p_2$. Let us take states $\varphi_1$ and $\varphi_2$ on $\mathscr{A}_1$ and $\mathscr{A}_2$, respectively. (There is  only one state on each algebra.) Then there is no common state extension of $\varphi_1$ and $\varphi_2$ over $\mathscr{A}$,  as this extension would assign value  $2=\varphi_1(p_1)+\varphi_2(p_2)$ to the unit of $\mathscr{A}$, which is not possible. Therefore $\mathscr{A}_1$ and $\mathscr{A}_2$ are not $C^*$-independent. However the normalized trace $\tau$ is a state on $\mathscr{A}$  whose restriction to $\mathscr{A}_1$ and $\mathscr{A}_2$ is $\frac{1}{2}$ multiple of the states $\varphi_1$ and $\varphi_2$, respectively. Thus the module independence holds with $m=\frac{1}{2}$. 
\end{example}
%=============================================%
However, when considering the configuration in which all algebras in question share common unit, the module  independence and $C^*$-independence coincide. Indeed,  suppose that $\mathscr{A}_1$ and $\mathscr{A}_2$ are $C^*$-subalgebras of $\mathscr{A}$ having the unit 1 of $\mathscr{A}$. Suppose that  $\mathscr{A}_1$ and $\mathscr{A}_2$ are module independent. Thanks to our assumption,  any state $\varphi$ on $\mathscr{A}$ that restricts to a multiple of a state $\varphi_1$ on $\mathscr{A}_1$ must in fact be equal to $\varphi_1$. Therefore condition (4) in Theorem~\ref{I1} is equivalent to $C^*$-independence. We have therefore the following consequence giving a new characterization of $C^*$-independence. 
%=============================================%
\begin{corollary}\label{I2}
	Let $\mathscr{A}_1$ and  $\mathscr{A}_2$ be unital
	$C^*$-subalgebras of a unital $C^*$-algebra  
	$\mathscr{A}$. Let us suppose that all algebras have the same unit. Then the following statements are equivalent:
	\begin{enumerate}
		\item  $\mathscr{A}_1$ and  $\mathscr{A}_2$ are module independent in 
		$\mathscr{A}$. 
		\item For each couple of states $\varphi_1\in S(\mathscr A_1)$ and  
		$\varphi_2\in S(\mathscr{A_2})$, there is a state $\varphi\in S(\mathscr A)$  with 
		
		\[ \mathscr{N}_{\varphi|\mathscr{A}_1} =        \mathscr{N}_{\varphi_1}            \] 
		and 
		\[ \mathscr{N}_{\varphi|\mathscr{A}_2} =        \mathscr{N}_{\varphi_2}   . \] 
		
		\item $\mathscr{A}_1$ and $\mathscr{A}_2$ are $C^*$-independent.  
\end{enumerate} 
\end{corollary}
%=============================================%
Now we shall generalize the above concepts to the context of $C^*$-modulus. We start with the natural definition of ternary space.
\begin{definition}
Let $\mathscr{E}$ be a Hilbert $\mathscr{A}$ module. A closed subspace $\mathscr{F}$ of $\mathscr{E}$ is said to be \emph{ternary} if $\mathscr{F}\langle\mathscr{F},\mathscr{F}\rangle \subseteq \mathscr{F}$. Then $\langle\mathscr{F},\mathscr{F}\rangle\,\langle\mathscr{F},\mathscr{F}\rangle\subseteq \langle\mathscr{F},\mathscr{F}\rangle$ and so $\langle\mathscr{F},\mathscr{F}\rangle$ is a $C^*$-subalgebra of $\mathscr{A}$ and $\mathscr{F}$ is a Hilbert $\langle\mathscr{F},\mathscr{F}\rangle$-module.
\end{definition}

For example, in the case that $\mathscr{E}$ is a $C^*$-algebra, $C^*$-subalgebras are ternary spaces. The same is true for closed subspaces of a Hilbert space as a Hilbert $\mathbb{C}$-module.

The following proposition gives a sufficient and necessary condition in order that $\langle\mathscr{F},\mathscr{F}\rangle$ is a $C^*$-algebra.
%=============================================%
\begin{proposition}
Let $\mathscr{A}$ be a $C^*$-algebra and let $\mathscr{E}$ be a Hilbert $C^*$-module. Let $\mathscr F $ be a closed ternary subspace of $\mathscr{E}$.
 Then $\langle \mathscr F,\mathscr F\rangle$ is a $C^*$-subalgebra of $\mathscr{A}$ and moreover 
$\mathscr F=\overline{\mathscr F\langle \mathscr F,\mathscr F\rangle}$.
\end{proposition}
\begin{proof}
Let $x_i,y_i\in \mathscr F$ for $i=1, 2$. It follows from 
$
\langle x_1,y_2\rangle\langle x_2,y_2\rangle=\langle x_1,y_1\langle x_2,y_2\rangle\rangle\in \langle \mathscr F,\mathscr F\rangle 
$ 
that $\mathscr{F}$ is closed under the multiplication. Evidently, $\langle \mathscr F,\mathscr F\rangle$ is closed under the involution, so it is a $C^*$-subalgebra of $\mathscr{A}$.

Let now $(u_i)$ be an approximate unit for it. It is enough to show that $ \mathscr F\langle \mathscr F,\mathscr F\rangle$ is dense in $\mathscr F$. Suppose $x\in \mathscr F$. We have 
\begin{align*}
\langle xu_i-x,xu_i-x\rangle=u_i\langle x,x\rangle u_i-\langle x,x\rangle u_i-u_i\langle x,x\rangle+\langle x,x\rangle\to 0.
\end{align*}
Hence $\lim_i xu_i =x$, as desired.
\end{proof}
%=============================================%
Now we introduce our key concept of Hilbert $C^*$-module independence.
\begin{definition}\label{Mi}
Let $\mathscr{A}$ be a unital $C^*$-algebra. Let $\mathscr{E}_i\subseteq \mathscr{E},\,\,i=1, 2$, be ternary spaces. Then $\mathscr{E}_1$ and $\mathscr{E}_2$ are said to be \emph{Hilbert $C^*$-module independent}, in short \emph{module independent}, if there are positive constants $m$ and $M$ such that for every $\varphi_i\in S(\mathscr \langle \mathscr{E}_i,\mathscr{E}_i\rangle),\,\,i=1, 2$, there exists a state $\varphi\in S(\mathscr{A})$ such that the following module independency condition holds:
\begin{align}\label{Hil1}
m\varphi_i(|x|)\leq \varphi(|x|) \leq M\varphi_i(|x|^2)^{\frac{1}{2}},\qquad \mbox{for all~}x\in \mathscr{E}_i, i=1, 2.
\end{align}
\end{definition}
%====================================================%
The following example shows that the above definition is not equivalent to the condition 
\[\|\langle x,y\rangle\|=\|x\|\,\|y\|,\qquad x\in \mathscr{E}_1, y\in \mathscr{E}_2,\]
which is the counterpart of the condition $\|ab\|=\|a\|\,\|b\|$ in the definition of $C^*$-independence.

\begin{example}
Let $(\mathscr{H},\langle\cdot,\cdot\rangle)$ be the complex Hilbert space and let $\mathscr K$ be a nontrivial closed subspace of $\mathscr H$. Let  $\mathscr{E}_1=\mathscr{K}$  and let $\mathscr E_2=\mathscr K^\perp$.

Since $\langle\mathscr E_1,\mathscr E_1\rangle=\langle \mathscr E_2,\mathscr E_2\rangle=\mathbb{C}$, it is clear that $\mathscr{E}_1$ and $\mathscr{E}_2$ are module independent with $m=M=1$. However, if $x\in \mathscr E_1$ and $y\in \mathscr E_2$ are nonzero, then $\|\langle x,y\rangle\|=0\neq \|x\|\|y\|$.
\end{example}
%=============================================%
Our next result shows that the roles of states $\varphi_1$ and $\varphi_2$ in the definition above can be played by pure states. To achieve our goal, we mimic some techniques in \cite{HAM2}.

\begin{theorem}\label{thpure}
Let $\mathscr{E}_1$ and $\mathscr{E}_2$ be ternary spaces of $\mathscr{E}$.
Then $\mathscr{E}_1$ and $\mathscr{E}_2$ are module independent if and only if for every pure states $\varphi_1$ and $\varphi_2$ on $\mathscr \langle \mathscr{E}_1,\mathscr{E}_1\rangle$ and $\langle \mathscr{E}_2,\mathscr{E}_2\rangle$, respectively, there is a state $\varphi$ on $\mathscr{A}$ such that \eqref{Hil1} holds for some fixed constants $m$ and $M$.
\end{theorem}
\begin{proof}
($\Longrightarrow$) Clear.\\
($\Longleftarrow$) For $i=1, 2$, let $\varphi_i$ be a state on $\mathscr{A}_i$, where $\mathscr{A}_i=\langle \mathscr{E}_i,\mathscr{E}_i\rangle$. We first assume that $\varphi_1$ and $\varphi_2$ are convex combinations of pure states of the form
\[
\varphi_1=\sum_{i=1}^n\alpha_i\psi_{1,i} \qquad \mbox{and}\qquad \quad\varphi_2=\sum_{j=1}^m\beta_j\psi_{2,j},
\]
where $\sum_{i=1}^n\alpha_i=1,\sum_{j=1}^m\beta_j=1,\alpha_i\geq0,\beta_j\geq 0$, and $\psi_{1,i}$ and $\psi_{2,j}$ are pure states for all $i=1, \ldots, n$ and $j=1, \ldots, m$. Without loss of generality, we assume $\alpha_i\neq0\neq \beta_j$ for all $i,j$. By the hypothesis, there are states $\rho_{ij}$ such that 
\begin{equation}\label{pure1}
m\psi_{1,i}(|x|) \leq \rho_{ij}(|x|)\leq M(\psi_{1,i}(|x|^2)^{\frac{1}{2}},\qquad x\in \mathscr{E}_1,
\end{equation}
\begin{equation}\label{pure2}
m\psi_{2,j}(|x|) \leq \rho_{ij} (|x|) \leq M(\psi_{2,j}(|x|^2)^{\frac{1}{2}}, \qquad x\in \mathscr{E}_2.
\end{equation}
Set
$
\varphi:=\sum_{i=1}^n\sum_{j=1}^m\alpha_i\beta_j{\rho_{ij}}.
$
We show that $\varphi$ is the desired state  in the  case considered. For $x\in \mathscr{E}_1$, it follows from \eqref{pure1} that
\begin{align*} 
\varphi(|x|) &=\sum_{i=1}^n\sum_{j=1}^m\alpha_i\beta_j\rho_{ij}(|x|)\nonumber\\
&\leq M\sum_{i=1}^n\sum_{j=1}^m\alpha_i\beta_j\psi_{1,i}(|x|^2)^{\frac{1}{2}}\nonumber\\
&= M\sum_{i=1}^n\alpha_i\psi_{1,i}(|x|^2)^{\frac{1}{2}}\nonumber\\
&\leq M\left(\sum_{i=1}^n\alpha_i\psi_{1,i}(|x|^2)\right)^{\frac{1}{2}}\quad ({\rm by~the~concavity~of~the~real~function~} f(t)=t^{\frac{1}{2}})\nonumber\\
&=\varphi_1(|x|^2)^{\frac{1}{2}}.
\end{align*}
Also, from \eqref{pure2}, we get
\begin{align*} 
\varphi(|x|)=\sum_{i=1}^n\sum_{j=1}^m\alpha_i\beta_j\rho_{ij}(|x|)\geq m\sum_{i=1}^n\sum_{j=1}^m\alpha_i\beta_j\psi_{1,i}(|x|)= m\sum_{i=1}^n\alpha_i\psi_{1,i}(|x|)\nonumber=m\varphi_1(|x|).
\end{align*}
Following the same arguments for $\varphi$ and $\varphi_2$, we observe that the module independency condition \eqref{Hil1} holds.

Now suppose that the $\varphi_1$ and $\varphi_2$ are arbitrary sates on $\mathscr{A}_1$ and $\mathscr{A}_2$. Since the pure states are the extreme points of $S(\mathscr{A}_i)$, the Krein--Milman theorem implies that there are nets $(\varphi_\alpha)$
and $(\varphi_\beta)$ consisting of convex combinations of pure states of $\mathscr{A}_i$ such that $\varphi_1$ and $\varphi_2$ are weak$^*$-limits of $(\varphi_\alpha) $ and $(\varphi_\beta)$, respectively. 

By the first part of the proof, assume that $\varphi_{\alpha\beta}$ is such that 
\begin{equation*} 
m\varphi_\alpha(|x|) \leq \varphi_{\alpha\beta} (|x|) \leq M\varphi_\alpha(|x|^2)^{\frac{1}{2}},\qquad x\in \mathscr{E}_1,
\end{equation*}
and
\begin{equation*} 
m\varphi_\beta(|x|) \leq \varphi_{\alpha\beta} (|x|) \leq M\varphi_\beta(|x|^2)^{\frac{1}{2}}, \qquad x\in \mathscr{E}_2.
\end{equation*}
By
the compactness of the state space of $\mathscr{A}$, we obtain a weak$^*$-limit point $\varphi$ of the net $(\varphi_{\alpha\beta})$. We show that $\varphi$ is the desired state. In fact, for any $x\in \mathscr{E}_1$, we have 
\begin{align*}
m\varphi_1(|x|)&=m\lim_\alpha \varphi_\alpha(|x|) \leq \lim_{\alpha,\beta} \varphi_{\alpha\beta}(|x|)=\varphi(|x|)\\&=\lim_{\alpha,\beta} \varphi_{\alpha\beta}(|x|)\leq M\lim_\alpha(\varphi_\alpha(|x|^2)^{\frac{1}{2}}=M(\varphi_1(|x|^2)^{\frac{1}{2}},
\end{align*}
where the limits are taken in the weak$^*$-topology. In a similar manner, we get 
\[
m\varphi_2(|x|)\leq \varphi(|x|)\leq M\varphi_2(|x|^2)^{\frac{1}{2}}.
\]
\end{proof}
%=============================================%
\begin{remark}
Let $\varphi_1\in S(\mathscr{A}_1)$ be a pure state and let $\varphi\in S(\mathscr{A})$ satisfy  \eqref{Hil1}. We claim that $\varphi$ is an extension of $\varphi_1$.
We set $\psi:=\varphi|_{\mathscr{A}_1}$ and prove that $\varphi_1=\psi$.
We first show that
\begin{equation}\label{nphi}
\mathcal{N}_{\varphi_1}=\mathcal{N}_{\psi}.
\end{equation}
To this end, let $b\in \mathcal{N}_{\varphi_1}$. Hence $\varphi_1(|b|^2)=0$. From \eqref{Hil1}, we infer that $\varphi(|b|)=0$. 
Therefore,
\[
\psi(b^*b)=\varphi(b^*b)=\varphi(|b|^2)\leq \|b\|\varphi(|b|)=0,
\]
whence $b\in \mathcal{N}_{\psi}$. \\
For the reverse inclusion, let $b\in \mathcal{N}_{\psi}$. Then $\varphi(b^*b)=\psi(b^*b)=0$. From \eqref{Hil1}, we conclude that $\varphi_1(b^*b)=0$. Hence $b\in \mathcal{N}_{\varphi_1}$.\\

Next, we note that $\ker(\varphi_1)=\mathcal{N}_{\varphi_1}+\mathcal{N}_{\varphi_1}^*$, since $\varphi_1$ is a pure state; see \cite[Theorem 5.3.4]{R}. Using \eqref{nphi}, we get 
\begin{equation}\label{ker}
\ker (\varphi_1)=\mathcal{N}_{\psi}+\mathcal{N}_{\psi}^*\subseteq \ker(\psi).
\end{equation}
To prove the inclusion in \eqref{ker}, assume that $a,b\in \mathcal{N}_{\psi}$. In virtue of \eqref{cs}, we have 
\begin{align*}
\psi\left((a+b^*)^*(a+b^*)\right)=\psi(a^*a)+\psi(a^*b^*)+\psi(ba^*)+\psi(bb^*)=0,
\end{align*}
whence, by \eqref{choi}, we get $\psi(a+b^*)=0$.\\
Now, \eqref{ker} ensures that $\psi=\alpha \varphi_1$. Since $\psi(1)=\varphi_1(1)=1$, we get $\alpha=1$ as desired. 
\end{remark}
%=============================================%
The next example shows that the above fact is not true for an arbitrary state $\varphi_1$.
\begin{example}
Let $\mathscr{A}=\mathscr{A}_1=\mathscr{A}_2=\mathbb{C}[0,1]$ and let $g:[0,1]\to \mathbb{R}^+$ be a nonconstant function (not necessary continuous) such that $g(t)\geq \frac{1}{2}$ and $\int_0^1g(t)dt=1$. Set
\[
m:=\frac{1}{2}\quad \mbox{and}\quad  M:=\left(\int_0^1g(t)^2dt\right)^{\frac{1}{2}}.
\]
Define the state $\varphi$ on $\mathscr{A}$ by 
\[
\varphi(f)=\int_0^1 g(t)f(t)dt.
\] 
Then $\varphi_1(f)=\int_0^1f(t)dt$ is a state on $\mathscr{A}$, and we have
\begin{align*}
\frac{1}{2}\varphi_1(|f|)&=\frac{1}{2}\int_0^1 |f(t)|dt\leq \int_0^1 g(t)|f(t)|dt\\
&\leq (\int_0^1g(t)^2)^\frac{1}{2}(\int_0^1|f(t)|^2)^\frac{1}{2}=M\varphi_1(|f|^2)^{\frac{1}{2}}.
\end{align*}
It is easy to see that $\varphi\neq \varphi_1$.
\end{example}
%=============================================%
It is shown in \cite[p. 343]{HAM3} that if $\mathscr{A}_1$ and $\mathscr{A}_2$ are $C^*$-independent, then $\mathscr{A}_1\cap \mathscr{A}_2= \mathbb{C}1$. Let us provide a different proof. Let $d\in \mathscr{A}_1\cap \mathscr{A}_2$ be a positive norm-one element. By the equivalence form of $C^*$-independence due to Florig and Summers \cite{FLS} (see Introduction), we have $\|dc\|=\|d\|\,\|c\|=\|1\|\,\|c\|=\|1c\|$ for all $c\in \mathscr{A}_1\cap \mathscr{A}_2$. It follows from \cite[Lemma 3.4]{LAN} that $d=1$, whence we reach the required result.\\
The next two results extend the above fact to the content of Hilbert $C^*$-modules. 

%=============================================%
\begin{proposition}\label{exte2}
Let $\mathscr{E}_1$ and $\mathscr{E}_2$ be module independent and let $\mathscr{E}_1\cap \mathscr{E}_2$ be a ternary space.  Then, $\langle\mathscr E_1\cap \mathscr E_2, \mathscr E_1\cap \mathscr E_2\rangle=\mathbb{C}|z|$ for some $z\in \mathscr{E}_1\cap \mathscr{E}_2$ with $\|z\|=1$.
\end{proposition}
\begin{proof}
Let $\varphi_1\in S(\langle \mathscr E_1,\mathscr E_1\rangle)$ be such that $\varphi_1(|z|^2)=1$. Let $x\in \mathscr E_1\cap \mathscr E_2$ and  $\varphi_2\in S(\langle \mathscr E_2,\mathscr E_2\rangle)$ be arbitrary. We first show that $\varphi_2(|z|^2|x|^2)=\varphi_2(|x|^2)$. 

By the module independency, there is $\varphi\in S(\mathscr A)$ satisfying \eqref{Hil1}. We have $\varphi_1(\big|\,x|z|^2-x\,\big|^2)=0$. Indeed, it follows from the determinacy of $|z|^2$ that
\begin{align*}
\varphi_1(\big|\,x|z|^2-x\,\big|^2)=\varphi_1(|z|^2|x|^2|z|^2)-\varphi_1(|x|^2|z|^2)-\varphi_1(|z|^2|x|^2)+\varphi_1(|x|^2)=0.
\end{align*}
Utilizing \eqref{Hil1}, we get $\varphi(\big|\,x|z|^2-x\,\big|)=0$. Again, from \eqref{Hil1}, we get $\varphi_2(\big|\,x|z|^2-x\,\big|)=0$. This gives that $\varphi_2(\big|\,x|z|^2-x\,\big|^2)=0$. So
\begin{align*}
\varphi_2(|x|^2|z|^2-|x|^2)=\varphi_2(\langle x,x|z|^2-x\rangle)\leq \varphi_2(|x|^2)^{\frac{1}{2}}\varphi_2(\big|x|z|^2-x\big|^2)^{\frac{1}{2}}=0.
\end{align*}
Hence $\varphi_2(|x|^2|z|^2)=\varphi_2(|x|^2)$. Since, $\varphi_2$ is arbitrary, we get $|x|^2|z|^2=|x|^2$. Since $\langle\mathscr E_1\cap \mathscr E_2, \mathscr E_1\cap \mathscr E_2\rangle$ is a $C^*$-algebra, $|z|^2 a=a|z|^2 =a$ for all $a\in \langle\mathscr E_1\cap \mathscr E_2, \mathscr E_1\cap \mathscr E_2\rangle$. Thus for all norm-one elements $z, w \in  \mathscr E_1\cap \mathscr E_2$, we have $|z|=|w|$. Therefore, $\langle\mathscr E_1\cap \mathscr E_2, \mathscr E_1\cap \mathscr E_2\rangle=\mathbb{C}|z|$ for some $z\in \mathscr{E}_1\cap \mathscr{E}_2$ with $\|z\|=1$.
\end{proof}
%=============================================%
We say that a $C^*$-algebra $\mathscr B$ has the quasi extension property relative to a containing $C^*$-algebra $\mathscr{A}$ if no pure state of $\mathscr{A}$ annihilates $\mathscr B$; see \cite{ARC}. 

\begin{proposition}\label{exte}
Let $\mathscr{E}_1$ and $\mathscr{E}_2$ be  module independent. Let either $\langle \mathscr{E}_1,\mathscr{E}_1\rangle $ or $\langle \mathscr{E}_2,\mathscr{E}_2\rangle$ have the quasi extension property relative to $\mathscr{A}$. If $z\in \mathscr{E}_1\cap \mathscr{E}_2$ with $\|z\|=1$, then $|z|=1$.
\end{proposition}
\begin{proof}
Let $\langle \mathscr{E}_2,\mathscr{E}_2\rangle$ have the quasi extension property. Suppose on contrary that $|z|\neq 1$. Note that $1-|z|^2$ is positive, since $\|z\|=1$. There are states $\varphi_1\in S(\langle \mathscr{E}_1,\mathscr{E}_1\rangle)$ and $\widetilde{\varphi}_2\in S(\mathscr{A})$ such that 
\[\varphi_1(|z|^2)=\|\,|z|^2\|=1\qquad \mbox{and}\qquad \widetilde{\varphi}_2(1-|z|^2)=\|1-|z|^2\|\neq 0.\]
Set $\varphi_2:=\widetilde{\varphi}_2|_{\langle \mathscr{E}_2,\mathscr{E}_2\rangle}$. By the module independency, there is a state $\varphi\in S(\mathscr{A})$ such that \eqref{Hil1} is valid. 
Due to that $\varphi_1$ is definite at $|z|^2$, that $\varphi_1\left(|z|^4\right)=\|z^4\|=\|z\|^4=1$ (see (iv) in Introduction), that $\left|z(1-|z|^2)\right|=|z|\left|(1-|z|^2)\right|$, and that $z(1-|z|^2)\in \mathscr{E}_1\cap \mathscr{E}_2$, we have
\begin{align*}
\varphi_1\left(\left|z(1-|z|^2)\right|^2\right) &=\varphi_1\left(|z|^2(1-|z|^2)^2\right)\\
&=\varphi_1\left(|z|^2\right)\varphi_1\left(1-2|z|^2+|z|^4\right)\\
&=\varphi_1\left(|z|^2\right)\left(1-2\varphi_1\left(|z|^2\right)\varphi_1\left(|z|^4\right)\right)=0. 
\end{align*}
Applying \eqref{Hil1}, we get $\varphi\left(\left|z(1-|z|^2)\right|\right)=0$. Again, from \eqref{Hil1}, we infer that
\[\varphi_2\left(\left|z(1-|z|^2)\right|\right)=0.\]
The later inequality and the fact that $\widetilde{\varphi}_2$ is definite at $1-|z|^2$ (see (iv) in Introduction) entail that
\begin{align*}
0=\varphi_2\left(\left|z(1-|z|^2)\right|\right)&= \varphi_2\left(|z|\left|(1-|z|^2)\right|\right)\\
&= \widetilde{\varphi}_2\left(|z|\left|(1-|z|^2)\right|\right)\\
 &=\widetilde{\varphi}_2(|z|) \widetilde{\varphi}_2(1-|z|^2).
\end{align*}
Since $\widetilde{\varphi}_2(1-|z|^2)\neq 0$, we have $\varphi_2 (|z|)=\widetilde{\varphi}_2(|z|)=0$, and so $\varphi_2 \left(|z|^2\right)\leq \|z\|\varphi_2(|z|) =0$. Another application of \eqref{Hil1} yields $\varphi(|z|) =0$ and so $\varphi_1(|z|) =0$. Thus $1=\varphi_1 \left(|z|^2\right)=0$ (see (iii) in Introduction), a contradiction.
\end{proof}
%=============================================%

Let $\mathscr{E}_1$ and $\mathscr{E}_2$ be ternary spaces of $\mathscr{E}$. In the following example, we show that if $\mathscr{E}_1$ and $\mathscr{E}_2$ are module independent, then it is not necessary that $\widetilde{\mathscr{A}}_1$ and $\widetilde{\mathscr{A}}_2$ are $C^*$-independent, where $\widetilde{\mathscr{A}}_i$ is the closed linear span of $\{\langle x,y\rangle: x, y\in \mathscr{E}_i\}\cup\{1\}, i=1, 2$.
%=============================================%
 The following example elaborates Example~\ref{EX10}.
\begin{example}
Let $(\mathscr{H},\langle\cdot,\cdot\rangle)$ be the complex Hilbert space with $\dim \mathscr{H}=2$. Let $\mathscr{A}$ be the $C^*$-algebra $\mathbb{B}(\mathscr{H})$ of all bounded linear operators on $\mathscr{H}$. Then $\mathscr{E}=\mathscr{H}$ is a Hilbert $C^*$-module under the right module action $x\cdot T=T^*(x)$ and the $\mathscr{A}$-inner product $\langle x,y\rangle'=x\otimes y$ in which $(x\otimes y)(z)=\langle z,y\rangle x$ for $x,y, z\in \mathscr{H}$. Let $\{e_1,e_2\}$ be an orthonormal basis of $\mathscr{H}$ and let $\mathscr{E}_i=\mathscr{H}_i$ be defined as the linear span of $e_i,\,\,i=1, 2$.

It is easy to see that if $\mathscr{A}_i=\langle \mathscr{H}_i,\mathscr{H}_i\rangle$, then $\mathscr{H}_i$ is an $\mathscr{A}_i$-module for each $i=1, 2$. We show that $\mathscr{H}_1$ and $\mathscr{H}_2$ are module independent with $M=\frac{1}{2}$. To achieve this purpose, let $\varphi_i$ be a state on $\mathscr{A}_i=(e_i\otimes e_i)\mathbb{C}$ for $i=1, 2$. Then $\varphi_1=w_{e_1}$ and $ \varphi_2=w_{e_2}$, where the vector state $w_x$ is defined by $w_x(T):=\langle Tx,x\rangle$. We consider the state $\varphi=\frac{1}{2}w_{e_1}+\frac{1}{2}w_{e_2}$. Then for any $x\in \mathscr{H}_i$, we have 
 \begin{align*}
\frac{1}{2}\varphi_i(|x|)\leq \varphi(|x|) \leq \varphi_i(|x|^2)^{\frac{1}{2}},\qquad \mbox{for all~}x\in \mathscr{E}_i, i=1, 2.
\end{align*}
 Note that $\widetilde{\mathscr{A}}_1=\widetilde{\mathscr{A}}_2=\left\{{\rm diag}(\alpha, \beta)\in \mathbb{M}_2(\mathbb{C}): \alpha, \beta\in\mathbb{C}\right\}$ are not $C^*$-independent, since $\widetilde{\mathscr{A}}_1\cap\widetilde{\mathscr{A}}_2\neq \mathbb{C}1$.
\end{example}
In the following example, we show that the quasi extension property is necessary for Proposition \ref{exte}. 
%=============================================%
\begin{example}
Let $\mathscr{A}$ contain a nontrivial projection $p$. Consider $\mathscr{E}=\mathscr{A}$ as a Hilbert $\mathscr{A}$-module. Then $\mathscr K=\{p\}$ is a ternary space. In addition, $\langle \mathscr K,\mathscr K\rangle$ does not satisfy the quasi extension property, since if $x\in \ker p$, then the vector state $\varphi=w_{x}$ annihilates $\mathscr K$. Also, it is easy to see that $\mathscr{E}_1=\mathscr{E}_2=\mathscr K$ are module independent, but $|p|\neq 1$.
\end{example}
%=============================================%

\begin{theorem}\label{eqa}
Let $\mathscr{E}_1$ and $\mathscr{E}_2$ be ternary spaces of $\mathscr{E}$ and let $z\in \mathscr{E}_1\cap \mathscr{E}_2$ be such that $\langle z,z\rangle=1$.
Then $\mathscr{E}_1$ and $\mathscr{E}_2$ are module independent if and only if $\mathscr{A}_1=\langle \mathscr{E}_1,\mathscr{E}_1\rangle$ and $\mathscr{A}_2=\langle \mathscr{E}_2,\mathscr{E}_2\rangle$ are $C^*$-independent. 
\end{theorem}
\begin{proof}
Let $\mathscr{E}_1$ and $\mathscr{E}_2$ be module independent. Then \eqref{Hil1} shows that $\mathscr{A}_1$ and $\mathscr{A}_2$ are module independent. Corollary \ref{I2} yields that $\mathscr{A}_1$ and $\mathscr{A}_1$ are $C^*$-independent. \\
Now suppose that $\mathscr{A}_1$ and $\mathscr{A}_2$ are $C^*$-independent. Let $\varphi_i\in S(\mathscr{A}_i)$ for $i=1, 2$. Then the common extension $\varphi$ satisfies \eqref{Hil1}. Hence $\mathscr{E}_1$ and $\mathscr{E}_2$ are module independent. 
\end{proof}

%=============================================%
%=============================================%
%=============================================%

\section{More characterizations of module independence and $C^*$-independence}

To prove the main result of this section, we need the result \cite[Theorem 9.1.4]{HAM3} stating that every pure state on a separable $C^*$-algebra $\mathscr{A}$ admits a determining element. A norm-one positive element $a\in\mathscr{A}$ is called \emph{determining element} for a pure state $\varphi$ if $\varphi$ is the only (pure) state on $\mathscr{A}$ such that $\varphi(a)=1$; see \cite[p. 277]{HAM3}.

\begin{theorem}\label{sys}
Let $\mathscr{E}_1$ and $\mathscr{E}_2$ be ternary spaces of $\mathscr{E}$ and let $\mathscr{A}_i=\langle \mathscr{E}_i,\mathscr{E}_i\rangle\,\,(i=1, 2)$.  Let $z_i\in \mathscr{E}_i$ be such that $|z_i|^2a_i=a_i$ for all $a_i\in \mathscr A_i\,\,(i=1,2)$. 
Then $\mathscr{E}_1$ and $\mathscr{E}_2$ are module independent if and only if there are positive constants $m'$ and $M'$ such that for every $x_i\in \mathscr{E}_i\,\,(i=1, 2)$ of norm-one elements and for every $\varphi_i \in S(\mathscr{A}_i)$ with $\varphi_i(|x_i|^2)=1\,\,(i=1, 2)$, there are $\widetilde{\varphi}_i\in S(\mathscr{A})$ for $i=1, 2$ such that
\begin{equation}\label{equi1}
\widetilde{\varphi}_1(|z|)\leq M'\varphi_2\left(|z|^2\right)^{\frac{1}{2}} \,\,(z\in \mathscr{E}_2)
~~\mbox{and}~~
m'\varphi_1(|z|)\leq\widetilde{\varphi}_1(|z|)\leq M'\varphi_1\left(|z|^2\right)^{\frac{1}{2}}\,\, (z\in \mathscr{E}_1)
\end{equation}
and
\begin{equation}\label{equi2}
\widetilde{\varphi}_2(|z|)\leq M'\varphi_1\left(|z|^2\right)^{\frac{1}{2}} \,\,(z\in \mathscr{E}_1)
~~\mbox{and}~~
m'\varphi_2(|z|)\leq\widetilde{\varphi}_2(|z|)\leq M'\varphi_2\left(|z|^2\right)^{\frac{1}{2}}\,\, (z\in \mathscr{E}_2).
\end{equation}
\end{theorem}

\begin{proof}
($\Longrightarrow$) For $i=1, 2$, let $x_i\in \mathscr{E}_i$ be norm-one elements and let $\varphi_i\in S(\mathscr{A}_i)$ be such that $\varphi_i\left(|x_i|^2\right)=1$. By Definition \ref{Mi}, we know that there is $\widetilde{\varphi}\in S(\mathscr A)$ such that \eqref{Hil1} is valid. If we set $\widetilde{\varphi}_1=\widetilde{\varphi}_2=\widetilde{\varphi}$, then \eqref{equi1}  and \eqref{equi2} are valid with $m'=m$ and $M'=M$. 

($\Longleftarrow$) To prove the assertion, it is enough to assume that $\varphi_1\in S(\mathscr{A}_1)$ and $\varphi_2\in S(\mathscr{A}_2)$ are arbitrary pure states; see Theorem \ref{thpure}. We shall show that there is a state $\varphi$ such that \eqref{Hil1} holds. 
As $\varphi_1$ and $\varphi_2$ are pure states, there are norm-one positive elements $a_1\in \mathscr{A}_1$ and $a_2\in  \mathscr{A}_2$ such that $\varphi_i(a_i)=1$, $i=1,2.$ 
This implies 
\[
\varphi_i(|z_ia_i^{\frac{1}{2}}|^2)=\varphi_i(\langle z_ia_i^{\frac{1}{2}},z_ia_i^{\frac{1}{2}}\rangle)=\varphi_i(a_i)=1\qquad (i=1, 2).
\]
Employing the hypothesis to the norm-one elements $z_1a_1^{\frac{1}{2}}\in \mathscr{E}_1$ and $z_2a_2^{\frac{1}{2}}\in \mathscr{E}_2$, there are $\widetilde{\varphi}_1,\widetilde{\varphi}_2\in S(\mathscr{A})$ satisfying \eqref{equi1} and \eqref{equi2}.
Define the state $\varphi\in S(\mathscr{A})$ by
\begin{align*}
\varphi=\frac{1}{2}\widetilde{\varphi}_1+\frac{1}{2}\widetilde{\varphi}_2.
\end{align*}
Then $\varphi$ satisfies condition \eqref{Hil1} with $m=\frac{m'}{2}$ and $M=M'$, since
\[\frac{m'}{2}\varphi_1(|z|)\leq \frac{1}{2}\widetilde{\varphi}_1(|z|)+ \frac{1}{2}\widetilde{\varphi}_2(|z|) \leq \frac{M'}{2}\varphi_1(|z|^2)^{\frac{1}{2}}+\frac{M'}{2}\varphi_1(|z|^2)^{\frac{1}{2}}=M'\varphi_1(|z|^2)^{\frac{1}{2}}\]
for all $z\in \mathscr{E}_1$, and
\[\frac{m'}{2}\varphi_2(|z|)\leq \frac{1}{2}\widetilde{\varphi}_2(|z|)+ \frac{1}{2}\widetilde{\varphi}_1(|z|) \leq \frac{M'}{2}\varphi_2(|z|^2)^{\frac{1}{2}}+\frac{M'}{2}\varphi_2(|z|^2)^{\frac{1}{2}}=M'\varphi_2(|z|^2)^{\frac{1}{2}}\]
for all $z\in \mathscr{E}_2$.

\end{proof}
%=============================================%
\begin{corollary}
Let $\mathscr{A}_1$ and $\mathscr{A}_2$ be $C^*$-subalgebras of a $C^*$-algebra $\mathscr{A}$ and let all algebras have the same identity. Then $\mathscr{A}_1$ and $\mathscr{A}_2$ are $C^*$-independent if and only if  for every positive norm-one elements $a_i\in \mathscr{A}_i\,\,(i=1, 2)$ and for every $\varphi_i \in S(\mathscr{A}_i)$ with $\varphi_i(a_i)=1\,\,(i=1, 2)$, there are extensions $\widetilde{\varphi}_i\in S(\mathscr{A})$ for $i=1, 2$ such that
\begin{equation}\label{equi11}
\widetilde{\varphi}_1(z)\leq \varphi_2\left(z^2\right)^{\frac{1}{2}} \,\,(z\in \mathscr{A}^+_2)
\end{equation}
and
\begin{equation}\label{equi22}
\widetilde{\varphi}_2(z)\leq \varphi_1\left(z^2\right)^{\frac{1}{2}} \,\,(z\in \mathscr{A}^+_1).
\end{equation}
\end{corollary}
\begin{proof}
Let $\mathscr{A}_1$ and $\mathscr{A}_2$ be $C^*$-independent. For $\varphi_i\in S(\mathscr{A}_i)\,\,(i=1, 2)$, let $\varphi\in S(\mathscr{A})$ be their common extension. Employing the Kadison inequality, we observe that $\widetilde{\varphi}_1=\widetilde{\varphi}_2=\varphi$ satisfy \eqref{equi11} and \eqref{equi22}.

For the reverse, put $\mathscr{E}_i=\mathscr{A}_i\,\,(i=1, 2)$. Then \eqref{equi1} and \eqref{equi2} are valid with $m'=1$ and $M'=1$. Thus, $\mathscr{A}_1$ and $\mathscr{A}_2$ are module independent. It follows from Theorem \ref{eqa} that $\mathscr{A}_1$ and $\mathscr{A}_2$ are $C^*$-independent. 
 \end{proof}
%=============================================% 

In the next result, we provide a characterization of module independence  dealing with one state.
 \begin{theorem}\label{Main}
Let $\mathscr{E}_1$ and $\mathscr{E}_2$ be ternary spaces of $\mathscr{E}$ and let $\mathscr{A}_i=\langle \mathscr{E}_i,\mathscr{E}_i\rangle$ for $i=1, 2$.  Let $z_i\in \mathscr{E}_i$ be such that $|z_i|a_i=a_i$ for all $a_i\in \mathscr A_i\,\,(i=1, 2)$. Then $\mathscr{E}_1$ and $\mathscr{E}_2$ are module independent if and only if there are positive constants $m$ and $M$ such that for every $x_i\in \mathscr{E}_i\,\,(i=1, 2)$ of norm-one elements,  there is $\varphi\in S(\mathscr{A})$  such that 
\begin{equation}\label{Miequ}
m\leq \varphi(|x_i|)=\varphi(|z_i|)\leq M\, \quad  (i=1,2).
\end{equation}
\end{theorem}
\begin{proof}
($\Longrightarrow$) For $i=1, 2$, let $x_i\in \mathscr{E}_i$ be norm-one and let $\varphi_i\in S(\mathscr{A}_i)$ be such that $\varphi_i\left(|x_i|\right)=1$. By the module independency, there is  $\varphi\in S(\mathscr A)$ such that \eqref{Hil1} is valid. Hence,
\[
m=m\varphi_i(|x_i|)\leq \varphi(|x_i|)\leq M\varphi_i(|x_i|^2)^{\frac{1}{2}}\leq M \varphi_i(|x_i|)^{\frac{1}{2}}\|\,|x_i|\,\|^{\frac{1}{2}}= M.
\]
Since $|z_i|$ acts as the identity of $\mathscr{A}_i$, we have $|z_i|^2=|z_i|$ and $\varphi_i(|z_i|)=1$. In addition, $\varphi_i\left(|x_i|^{\frac{1}{2}}\right)=\varphi_i\left(|x_i|\right)=1$. Therefore, 
\[\varphi_i(\big|z_i|x_i|^{\frac{1}{2}}-z_i\big|^2)=\varphi_i\left(|x_i||z_i|^2-|x_i|^{\frac{1}{2}}|z_i|^2-|z_i|^2|x_i|^{\frac{1}{2}}+|z_i|^2\right)=0.\]
Hence, from \eqref{Hil1} and the fact that $|z_i|x_i|-z_i|=|z_i|(1-|x_i|)$, 
 we have 
\begin{align}\label{eq33}
0=\varphi(\big|z_i|x_i|-z_i\big|)=\varphi(|z_i|(1-|x_i|))=\varphi(|z_i|)-\varphi(|x_i|).
\end{align}
In other words, $\varphi(|x_i|)=\varphi(|z_i|)$.\\

($\Longleftarrow$) To prove the assertion, it is enough to assume that $\varphi_1\in S(\mathscr{A}_1)$ and $\varphi_2\in S(\mathscr{A}_2)$ are arbitrary pure states; see Theorem \ref{thpure}. We   show that there is a state $\varphi$ such that \eqref{Hil1} holds. 

First, we suppose that $\mathscr{A}_1$ and $\mathscr{A}_2$ are separable. According to \cite[Corollary 9.1.4]{HAM3}, there are determining norm-one positive elements $a_i\in \mathscr{A}_i$ for $\varphi_i,\,\,i=1, 2$. The determinacy $a_i$ for $\varphi_i$ implies that 
\[
\varphi_i(|z_ia_i|^2)=\varphi_i(\langle z_ia_i,z_ia_i\rangle)=\varphi_i(a_i^2)=1\qquad (i=1, 2),
\]
since $\varphi_i(a_i)=1$. Employing the hypothesis to the norm-one elements $x_i=z_ia_i\in \mathscr{E}_i\,\,(i=1, 2)$, we infer that there is $\varphi\in S(\mathscr{A})$ satisfying \eqref{Miequ}. Hence,
\begin{equation}\label{eq11}
\varphi(a_i)=\varphi(|x_i|)=\varphi(|z_i|).
\end{equation}
It follows from the inequality at \eqref{Miequ} that $\varphi|_{\mathscr A_i}\neq 0$ for $i=1,2$. Set 
\[
\widetilde{\varphi_i}:=\frac{\varphi|_{\mathscr A_i}}{\|\varphi|_{\mathscr A_i}\|} \quad (i=1,2).
\]
Then $\widetilde{\varphi_i}$ is a state on $\mathscr A_i$. Equality \eqref{eq11} yields that
\[
\widetilde{\varphi_i}(a_i)=\frac{\varphi(a_i)}{\|\varphi|_{\mathscr A_i}\|}=\frac{\varphi(|z_i|)}{\varphi(|z_i|)}=1.
\] 
By the determinacy of $a_i$ for $\varphi_i$, we get $\widetilde{\varphi_i}=\varphi_i$.  For any $x\in \mathscr E_i$, we therefore have 
\begin{align*}
m\varphi_i(|x|)&\leq  \varphi(a_i)  \varphi_i(|x|)\leq \|\varphi|_{\mathscr A_i}\|\,\varphi_i(|x|)\\
&=\varphi(|x|)=\|\varphi|_{\mathscr A_i}\|\,\varphi_i(|x|)= \varphi(|z_i|)\,\varphi_i(|x|)\\
&=\varphi(a_i)\,\varphi_i(|x|)\leq M\varphi_i(|x|^2)^{\frac{1}{2}}.
\end{align*}
Here we use the fact that $\varphi|_{\mathscr A_i}$ is a positive linear functional and so its norm is equal to its value in the identity $|z_i|$ of $\mathscr A_i$. 

 Suppose that $\mathscr{A}_1$ and $\mathscr{A}_2$ are not necessarily separable. Denote by $\mathscr{S}$ the system of all
finite nonempty subsets of $\mathscr{E}_1\cup \mathscr{E}_2$ containing $z$. For any $S\in \mathscr{S}$, we put
\[
F_S=\{\varphi\in \mathcal{S}(\mathscr{A}):\varphi \mbox{~satisfies~}\eqref{Hil1}\}.
\]
Since for any $S\in \mathscr{S}$, the $C^*$-subalgebras $\langle S\cap \mathscr{E}_1,S\cap \mathscr{E}_1\rangle$ and $\langle S\cap \mathscr{E}_2,S\cap \mathscr{E}_2\rangle$ are separable, there is a state $\varphi_S\in \mathcal{S}(A)$ satisfying \eqref{Hil1}. In other words, the sets $F_S$'s are nonempty and closed in the compact space $\mathcal{S}(\mathscr{A})$. Moreover, the inclusion
\[
F_{\cup_{i=1}^nS_i}\subseteq \cap_{i=1}^nF_{S_i}
\] 
shows that the system $(F_S)_{S\in \mathscr{S}}$ satisfies the finite intersection property, and so $F=\cap_{S\in \mathscr{S}}F_S\neq \emptyset$. Thus, any $\varphi\in F$ is the desired state satisfying \eqref{Hil1}.

%Secondly, suppose that $\mathscr{A}_1$ and $\mathscr{A}_2$ are not necessarily separable. Employing the same reasoning as in the last part of the proof of Theorem \ref{sys} we arrive at the required result.
\end{proof}
%=============================================%

\begin{theorem}
Let $\mathscr{E}_1$ and $\mathscr{E}_2$ be ternary spaces of $\mathscr{E}$ and let $\mathscr{A}_i=\langle \mathscr{E}_i,\mathscr{E}_i\rangle$ for $i=1, 2$.  Let $z_i\in \mathscr{E}_i$ be such that $|z_i|a_i=a_i$ for all $a_i\in \mathscr A_i\,\,(i=1)$. 
Then $\mathscr{E}_1$ and $\mathscr{E}_2$ are module independent if and only if there are positive constants $m$ and $M$ such that for every $x_i\in \mathscr{E}_i\,\,(i=1, 2)$ of norm-one elements, there is $\varphi\in S(\mathscr{A})$  such that $\frac{\varphi}{\varphi(|z_i|)}$ is definite at $|x_i|$ and 
\begin{equation}\label{Miequ2}
 m\leq \varphi(|x_1|)=\varphi(|x_1||x_2|)=\varphi(|x_2|)\leq M.
\end{equation}
\end{theorem}
 \begin{proof}
 Let  $\mathscr{E}_1$ and $\mathscr{E}_2$ be module independent. It follows from Theorem \ref{Main} that there exists $\varphi\in S(\mathscr A)$ such that \eqref{Miequ} is valid. At first, we show that $\psi=\frac{\varphi}{\varphi(|z_i|)}$ is definite at $|x_i|$, 
 but it follows immediately from  the fact that 
 \[ \psi(|x_i|)= \frac{\varphi(|x_i|)}{\varphi(|z_i|)} =1.    \]
Thus,
 \begin{align}\label{Mieq3}
 \frac{\varphi(|x_1||x_2|)}{\varphi(|z_i|)}=\frac{\varphi(|x_1|)}{\varphi(|z_i|)}\frac{\varphi(|x_2|)}{\varphi(|z_i|)} \quad(\mbox{for\quad}i=1,2).
 \end{align}
Applying \eqref{Mieq3} for $i=1$ and again for $i=2$ and employing \eqref{Miequ}, we get 
 \begin{equation*} 
 m\leq \varphi(|x_1|)=\varphi(|x_1||x_2|)=\varphi(|x_2|)\leq M.
 \end{equation*}
 
 For the reverse statement, let $x_i\in \mathscr E_i$ be norm-one. Let $\varphi\in S(\mathscr A)$ be definite at $x_i$ and let \eqref{Miequ2} be valid. We show that \eqref{Miequ} is valid. In fact,  
 \begin{align*}\label{eq11}
\frac{\varphi(|x_i|)}{\varphi(|z_i|)}&=\frac{\varphi(|x_1||x_2|)}{\varphi(|z_i|)}\qquad\qquad(\mbox{by \eqref{Miequ2}})\\
&=\frac{\varphi(|x_1|)}{\varphi(|z_1|)}\frac{\varphi(|x_2|)}{\varphi(|z_1|)}\qquad(\frac{\varphi}{\varphi(|z_i|)} \mbox{~is~ definite~at~} x_i)\\
&=\left(\frac{\varphi(|x_i|)}{\varphi(|z_i|)}\right)^2\qquad\qquad(\mbox{by \eqref{Miequ2}}).
\end{align*}
Hence $\varphi(|x_i|)=\varphi(|z_i|)$. Now, Theorem~\ref {Main} ensures that $\mathscr{E}_1$ and $\mathscr{E}_2$ are module independent.
\end{proof}

%=============================================%
%=============================================%
%=============================================%

\section{Deforming a Hilbert $C^*$-module to a $C^*$-algebra}

Deforming a space is a useful technique to investigate some properties of the space; for example, see \cite{NAJ}.

Let $\mathscr{E}$ be a Hilbert $\mathscr{A}$-module. Let $z_0\in \mathscr{E}$ be such that $\langle z_0,z_0\rangle=1$. Then $\mathscr{E}$ is orthogonally decomposed into two summands, the one generated by $z_0$ and its orthogonal complement. Moreover, the $C^*$-algebra $\mathscr{A}$ is canonically isomorphic to the module generated by $z_0$ via the isomorphism $a\mapsto z_0a,\,\,a\in A$, or equivalently, to the quotient module $[\mathscr{E}]:=\frac{\mathscr{E}}{\{z_0\}^\perp}$. For the convenience, we denote $x+\{z_0\}^\perp$ by $[x]$ for $x\in \mathscr{E}$. Since $\langle x-z_0\langle z_0,x\rangle,z_0\rangle=0$, we have 
\begin{equation}\label{equot}
[x]=[z_0\langle z_0,x\rangle] \qquad (x\in \mathscr{E}).
\end{equation}

 Consider the natural operations of addition and scalar multiplication on $[\mathscr{E}]$, and define the well-defined multiplication $\circ$ as follows:
\[
\circ:[\mathscr{E}]\times [\mathscr{E}]\to [\mathscr{E}]\,,\quad [x]\circ[y]=[z_0\langle z_0,x\rangle\langle z_0,y\rangle].
\]
Also the involution on $[\mathscr{E}]$ is defined by 
\[
[x]^\sharp=[z_0\langle x, z_0\rangle].
\]
In virtue of \eqref{equot}, we observe that $[z_0]$ is the identity of $[\mathscr{E}]$. 
Put the following norm on $[\mathscr{E}]$:
\[
\|\,[x]\,\|=\|\langle z_0,x\rangle\|.
\]
With this notation, one can easily verify that $([\mathscr{E}],\circ,\sharp,\|\cdot\|)$ is a $C^*$-algebra.

Let $\mathscr F$ be a ternary space. Then $[\mathscr F]=\{[x]:x\in \mathscr F\}$ is a $C^*$-subalgebra of $[\mathscr{E}]$. Let $\varphi\in S([\mathscr F])$. We define the functional $\widetilde{\varphi}$ by 
\[
\widetilde{\varphi}(a)=\varphi([z_0a]).
\]
Due to $\widetilde{\varphi}(b^*b)=\varphi([z_0b^*b]=\varphi([z_0b]^\#\circ[z_0b])\geq 0$ for all $b\in \langle \mathscr F,\mathscr F\rangle$ and $\widetilde{\varphi}(1)=\varphi([z_0])=1$, we arrive at $\widetilde{\varphi}\in S(\langle \mathscr F,\mathscr F\rangle)$.

Also it is seen that if $\psi\in S(\langle \mathscr F,\mathscr F\rangle) $, then by defining 
$\hat{\psi}([x])=\psi(\langle z_0, x\rangle)$, we have $\hat{\psi}\in S([\mathscr F])$. Indeed, if $[x]\geq 0$, then there is $y\in \mathscr F$ such that $[x]=[y]^\sharp\circ [y]=[z_0|\langle z_0,y\rangle|^2]$. So, we have 
\[
\hat{\psi}([x])=\psi(\langle z_0,z_0|\langle z_0,y\rangle|^2\rangle)=\psi(|\langle z_0,y\rangle|^2)\geq 0.
\]
In addition, $\hat{\psi}([z_0])=\psi(1)=1$. 

Furthermore,  
\[
\widetilde{\hat{\psi}}(a)=\hat{\psi}([z_0a])=\psi(\langle z_0,z_0a\rangle)=\psi(a)
\]
for all $a\in \langle \mathscr F,\mathscr F\rangle$. Similarly, $\hat{\widetilde{\varphi}}=\varphi$.
 Hence 
\begin{equation}
S([\mathscr{F}])\to S(\langle \mathscr{F},\mathscr{F}\rangle)\,,\quad
\varphi\mapsto\widetilde{\varphi}
\end{equation}
is an algebraic isomorphism. 
%=============================================%
\begin{theorem}\label{orth}
Let $\mathscr{E}_1$ and $\mathscr{E}_2$ be ternary spaces of $\mathscr{E}$ and let $z_0\in \mathscr{E}_1\cap \mathscr{E}_2$ be such that $\langle z_0,z_0\rangle=1$.
Then $\mathscr{E}_1$ and $\mathscr{E}_2$ are module independent if and only if $C^*$-subalgebras $[\mathscr{E}_1]$ and $[\mathscr{E}_2]$ of $[\mathscr{E}]$ are $C^*$-independent.
\end{theorem}
\begin{proof}
Suppose that $\mathscr{E}_1$ and $\mathscr{E}_2$ are module independent with $m\leq 1\leq M$. We shall show that $[\mathscr{E}_i]$ are $C^*$-independent. Let $\varphi_i\in S([\mathscr{E}_i])$ and let $\widetilde{\varphi}_i\in S(\mathscr{E}_i,\mathscr{E}_i\rangle)$, for $i=1,2$, be the corresponding states constructed in the preceding of this theorem. Assume that $\widetilde{\varphi}$ is the state in $S(\mathscr{A})$ satisfying \eqref{Hil1}. Set $\varphi:=\hat{\widetilde{\varphi}}$. Since $[x]=[z_0\langle z_0,x\rangle]$ for all $x\in \mathscr{E}_i$, we have 
\begin{align*}
m\varphi_i(|[x]|)&=m\varphi_i(([x]^\sharp \circ [x])^{\frac{1}{2}})\\
&=m\varphi_i([z_0|\langle z_0,x\rangle|])\\
&=m\widetilde{\varphi}_i(|\langle z_0,x\rangle|)\\
&\leq \widetilde{\varphi}(|\langle z_0,x\rangle|)\\
&= \varphi([z_0|\langle z_0,x\rangle|])\\
&=\varphi(|[x]|)\\
&=\widetilde{\varphi}(|\langle z_0,x\rangle |)\\
&\leq M\widetilde{\varphi}_i(|\langle z_0,x\rangle|^2)^{\frac{1}{2}}\\
&=M\varphi_i([z_0|\langle z_0,x\rangle|^2])^{\frac{1}{2}}\\
&=M\varphi_i(|[x]|^2)^{\frac{1}{2}}\qquad (i=1, 2).
\end{align*}
This shows that $[\mathscr{E}_i]$ are module independent. It follows from Corollary \ref{I2} that $[\mathscr{E}_i]$ are $C^*$-independent. 

For the reverse, let $[\mathscr{E}_i]$ be $C^*$-independent and let $\mathscr \varphi_i\in S(\langle \mathscr{E}_i,\mathscr{E}_i\rangle)$ for $i=1,2$. Since $\hat{\varphi}_i\in S([\mathscr{E}_i])$, by $C^*$-independency, there is a common extension $\varphi\in S([\mathscr{E}])$. Then it is easy to see that $\widetilde{\varphi}\in S(\mathscr{A})$ satisfies \eqref{Hil1}. 
\end{proof}
%=============================================%

\begin{corollary}\label{FFSS}
Let $\mathscr{E}_1$ and $\mathscr{E}_2$ be ternary spaces of $\mathscr{E}$ and let $z_0\in \mathscr{E}_1\cap \mathscr{E}_2$ be such that $\langle z_0,z_0\rangle=1$. Then $\mathscr{E}_1$ and $\mathscr{E}_2$ are module independent if and only if 
\begin{equation}\label{inde}
\|\langle x,z_0\rangle\langle y,z_0\rangle\|=\|\langle x,z_0\rangle\|\,\|\langle y,z_0\rangle\|\quad (x\in \mathscr{E}_1, y\in \mathscr{E}_2).
\end{equation}
\end{corollary}
\begin{proof}
Note that \eqref{inde} is equivalent to
\[
\|\,[x]\circ [y]\,\|=\|\,[x]\,\|\,\|\,[y]\,\|.
\]
Utilizing \cite[Theorem 11.2.5]{HAM3}, we deduce that $[\mathscr{E}_1]$ and $[\mathscr{E}_2]$ are $C^*$-independence. In virtue of Theorem \ref{orth}, this is equivalent to the module independency of $\mathscr{E}_1$ and $\mathscr{E}_2$.
\end{proof}
%=============================================%

\begin{corollary} \cite{FLS} Let $\mathscr{A}_1$ and $\mathscr{A}_2$ be $C^*$-subalgebras of a $C^*$-algebra $\mathscr{A}$ and let all algebras have the same identity $1$.  Then $\mathscr{A}_1$ and $\mathscr{A}_2$ are $C^*$-independent if and only if $\|ab\|=\|a\|\,\|b\|$ for all $a\in\mathscr{A}_1$ and $b\in \mathscr{A}_2$.
\end{corollary}
\begin{proof}
It follows from Corollaries \ref{I2} and \ref{FFSS} with $z_0=1$.
\end{proof}

%=============================================%
%=============================================%
%=============================================%
\section{Further examples}

In this section, we go further and show that how module independence differs from  $C^*$-independence in some directions. We start this section with following proposition.  

\begin{proposition}
Let $\mathscr E_i\subseteq\mathscr A$, $i=1,2$, be self-modules such that $\dim\mathscr E_i=1$ and $\dim(\langle\mathscr E_i,\mathscr E_i\rangle)=1$, $i=1,2$. Then they are module independent.
\end{proposition}
\begin{proof}
Let $x\in\mathscr E_1$, $x\neq 0$. Then $xx^*x=x\langle x,x\rangle=\lambda x$ for some $\lambda\in\mathbb C$. Hence $x$ is a partial isometry up to a scalar multiple. Let $u,v\in\mathscr A$ be partial isometries generating $\mathscr E_1$ and $\mathscr E_2$, respectively. Set $p:=u^*u$ and $q:=v^*v$.  

Let $\varphi_i$ be the unique state on $\langle\mathscr E_i,\mathscr E_i\rangle$, $i=1,2$. Then $\varphi_1(p)=1$ and  $\varphi_2(q)=1$. Let $\widetilde{\varphi}_i$ be the state that extends $\varphi_i$ to the whole $\mathscr A$, and set 
$$\varphi(a):=\frac{1}{2}(\widetilde{\varphi}_1(pap)+\widetilde{\varphi}_2(qaq)).$$

If $x\in\mathscr E_1$, $\|x\|=1$, then $|x|=p$. Hence $\varphi_1(|x|)=\varphi(|x|^2)^{\frac{1}{2}}=1$, while $\varphi(|x|)=\varphi(p)\geq\frac{1}{2}\widetilde{\varphi}_1(p)=\frac{1}{2}$ and $\varphi(|x|)\leq \frac{1}{2}\varphi_1(p)+\frac{1}{2}\widetilde{\varphi}_2(qpq)\leq \frac{1}{2}+\frac{1}{2}\|qpq\|\leq 1$. Similar estimates hold for $x\in\mathscr E_2$. Thus, $\mathscr E_1$ and $\mathscr E_2$ are module independent with $m=\frac{1}{2}$ and $M=1$.    

\end{proof}

In dimensions greater than one, more effects can be noted. In particular, we show that the module independence (and the $C^*$-independence) is not stable under small perturbations.

\begin{example}
Let $\mathscr A=M_4$. Set 
$$
p_t=\left(\begin{matrix}1&0&0&0\\0&\cos^2\frac{\pi t}{2}&\sin \frac{\pi t}{2}\cos \frac{\pi t}{2}&0\\0&\sin \frac{\pi t}{2}\cos \frac{\pi t}{2}&\sin^2\frac{\pi t}{2}&0\\0&0&0&0\end{matrix}\right); %\quad 
p_0=\left(\begin{matrix}1&0&0&0\\0&1&0&0\\0&0&0&0\\0&0&0&0\end{matrix}\right);%\quad
p_1=\left(\begin{matrix}1&0&0&0\\0&0&0&0\\0&0&1&0\\0&0&0&0\end{matrix}\right).
$$ 
Let $\mathscr A_t\subseteq\mathscr A$ be the two-dimensional commutative subalgebra generated by $p_t$ and $1-p_t$. Let $\varphi_1$ be a state on $\mathscr A_0$ and let  $\varphi_2$ be a state on $A_t$. They are determined by their values on $p_0$ and $p_t$, respectively. Set $\varphi_1(p_0):=\alpha$ and $\varphi_2(p_t):=\beta$. Let $\varphi$ denote a state on $\mathscr A$.    

First, consider the case $t=0$. Then the two copies of $\mathscr A_0$ are not module independent. Indeed, take $\alpha=0$ and $\beta=1$. Then there should exist $m,M>0$ and a state $\varphi$  such that 
\begin{equation}\label{1}
0=m\varphi_1(p_0)\leq\varphi(p_0)\leq M\varphi_1(p_0)^{\frac{1}{2}}=0
\end{equation}
 and 
\begin{equation}\label{2}
m=m\varphi_2(p_0)\leq\varphi(p_0)\leq M\varphi_2(p_0)^{\frac{1}{2}}=M, 
\end{equation}
but (\ref{1}) implies that $\varphi(p_0)=0$, while (\ref{2}) requires that $\varphi(p_0)$ is separated from 0. 

Now, consider the case $t=1$. Then $\mathscr A_0$ and $\mathscr A_1$ are module independent with $m=M=1$. Indeed, set 
$$
\varphi\left(\begin{matrix}a_{11}&a_{12}&a_{13}&a_{14}\\a_{21}&a_{22}&a_{23}&a_{24}\\a_{31}&a_{32}&a_{33}&a_{34}\\
a_{41}&a_{42}&a_{43}&a_{44}\end{matrix}\right):=\alpha\beta a_{11}+\alpha(1-\beta)a_{22}+(1-\alpha)\beta a_{33}+(1-\alpha)(1-\beta)a_{44}.
$$
Then 
$$
\alpha=\varphi_1(p_0)=\varphi(p_0)=\alpha\beta+\alpha(1-\beta)\leq\varphi_1(p_0)^{\frac{1}{2}}=\sqrt{\alpha},
$$
$$
1-\alpha=\varphi_1(1-p_0)=\varphi(1-p_0)=(1-\alpha)\beta+(1-\alpha)(1-\beta)\leq\varphi_1(1-p_0)^{\frac{1}{2}}=\sqrt{1-\alpha},
$$
$$
\beta=\varphi_2(p_1)=\varphi(p_1)=\alpha\beta+(1-\alpha)\beta=\beta\leq\varphi_2(p_1)^{\frac{1}{2}}=\sqrt{\beta},
$$
$$
1-\beta=\varphi_2(1-p_1)=\varphi(1-p_1)=\alpha(1-\beta)+(1-\alpha)(1-\beta)=1-\beta\leq\varphi_2(1-p_1)=\sqrt{1-\beta}.
$$

Now, fix $t\in(0,1)$, and take $\alpha=1$ and $\beta=0$. It follows from $\varphi(p_t)\leq M\varphi_2(p_t)^{\frac{1}{2}}=0$ that $\varphi(p_t)=0$. Let $e_1=\operatorname{diag}(1,0,0,0)$ and let $q_t=p_t-e_1$. As both $e_1$ and $q_t$ are positive, $\varphi(p_t)=0$ implies that $\varphi(e_1)=0$ and $\varphi(q_t)=0$. As $\varphi(1-p_0)\leq M\varphi_1(1-p_0)=0$, we have $\varphi(p_0)=1$ (as $\varphi(1)=1$). Then $\varphi(e_2)=1$, where $e_2=\operatorname{diag}(0,1,0,0)$. This implies that $\varphi(A)=a_{22}$, where $A=(a_{ij})_{i,j=1}^4$, which contradicts $\varphi(q_t)=0$. Thus, for the given states $\varphi_1$ and $\varphi_2$, there is no state $\varphi$ satisfying the condition of module independence. Therefore  the pair $(\mathscr A_0,\mathscr A_t)$ is not module independent for any $t\in[0,1)$, while the pair $(\mathscr A_0,\mathscr A_1)$ is module independent. In conclusion,  module independence is not stable   under small perturbations.
\end{example}
\medskip

\noindent \textit{Conflict of Interest Statement.} On behalf of all authors, the corresponding author states that there is no conflict of interest.\\

\noindent\textit{Data Availability Statement.}  Data sharing not applicable to this article as no datasets were generated or analysed during the current study.

\medskip
%=============================================%
%=============================================%
%=============================================%

\end{document}